\documentclass[12pt]{article}
\newcommand{\R}{\mathbb{R}}
\usepackage{amssymb,latexsym,amsmath,epsfig,pdfsync, amsthm} 

\usepackage{amsmath,amssymb,enumerate,bbm}
\newtheorem{theorem}{Theorem}[section]

\newtheorem{lemma}[theorem]{Lemma}

\numberwithin{equation}{section}

\begin{document}

\title{The local criteria for blowup of the Dullin-Gottwald-Holm equation and the two-component Dullin-Gottwald-Holm system}
\author{Duc-Trung Hoang\\
Department of Mathematics\\Ecole Normale Sup\'erieure de Lyon, France\\
trung.hoang-duc@ens-lyon.fr}

\maketitle

\abstract{We investigate wave breaking criteria for  the  Dullin-Gottwald-Holm equation and the two-component  Dullin-Gottwald-Holm system. We establish a new blow-up criterion for the general case $\gamma+c_0\alpha^2 \geq 0$ involving local-in-space conditions on the initial data.}

% \begin{center}
% \uppercase{\bf The local criteria for blowup of the Dullin-Gottwald-Holm
% equation}
% \vskip 20pt
% {\bf Hoang Duc Trung}\\
% {\smallit Department of Mathematics\\Ecole Normale Superieure de Lyon, France}\\
% {\tt trung.hoang-duc@ens-lyon.fr}
% \vskip 10pt
% \end{center}

\section{Introduction}
The DGH equation is a nonlinear dispersive equation, modelling the propagation of undirectional shallow waters over a flat bottom. 
It was proposed by  Dullin, Gottwald and Holm \cite{dullin-gottwald-holm} in 2001 and derived from the water wave theory  by using the method of asymptotic analysis and a near identity normal transformation.
The equation reads  
\begin{equation} \label{s1}
\begin{cases} 
u_t - \alpha^2u_{txx} + c_0u_x + 3uu_x +\gamma u_{xxx}= \alpha^2(2u_xu_{xx}+uu_{xxx}) , \quad \quad t > 0, \quad x\in \mathbb{R}, \\ 
u(0,x) = u_0.
\end{cases}
\end{equation}
Here, $u$ stands for a fluid velocity in the $x$ direction and 
$c_0, \alpha^2, \gamma$ are physical parameters: $\alpha^2$ and ${\gamma}/{c}$ are squares of length scales. 
The constant $c_0 = \sqrt{gh}$ is the critical shallow water speed, while $h$ is the mean fluid depth and $g$
is the gravitational constant. In \cite{dullin-gottwald-holm}, the authors proved that the phase speed lies in the band $(-\frac{\gamma}{\alpha^2}, c_0 )$ and longer linear wave are faster
provided that $\gamma + c_0\alpha^2 \geq 0$. 

Let $m = u - \alpha^2u_{xx}$ be the momentum variable. 
Equation~\eqref{s1} can be rewritten in terms of the momentum as
\begin{equation} \label{s2}
\begin{cases}  
m_t + c_0u_x + um_x+ 2mu_x = -\gamma u_{xxx}, \quad \quad t > 0, \quad x\in \mathbb{R}, \\
m(0,x) = m_0.
\end{cases} 
\end{equation}
We denote by 
\[
p(x) = \frac{1}{2\alpha}e^{-|\frac{x}{\alpha}|}
\] 
the Green function for the operator $Q := (1-\alpha^2\partial_x^2)^{-1}$, in a such way that 
$Qf = (1-\alpha^2\partial_x^2)^{-1}f = p*f$ for $f \in L_2(\mathbb{R})$. 
The convolution relation $p*m = u$ allows to recover $u$ from~$m$.

The DGH equation can also be reformulated as a quasi-linear evolution equation of hyperbolic type :
\begin{equation}\label{s5}
\begin{cases}  
u_t + (u-\frac{\gamma}{\alpha^2})u_{x}= -\partial_xp*(\frac{\alpha^2}{2}u_x^2+u^2 + (c_0 + \frac{\gamma}{\alpha^2})u) , \quad \quad t > 0, \quad & x\in \mathbb{R} \\ 
u(x,0) = u_0(x), & x\in \mathbb{R}.
\end{cases} 
\end{equation}

When $\gamma =0$ and $\alpha = 1$  the system (\ref{s1}) boils down  to the Camassa-Holm equation (the dispersionless
Camassa--Holm equation if in addition $c_0=0$), which was derived by Camassa and Holm \cite{camassa-holm} by approximating directly the Hamiltonian for the Euler equations for an irrotational flow in the shallow water regime. 
In the past decades, a considerable number of papers investigated various properties of Camassa-Holm equation such as local well-posedness, blow-up phenomena, persistence properties of solutions, global existence of weak solutions. 
%It is very interesting to know such properties can be applied to the DGH equation. \\

Similarly to the Camassa-Holm equation, equation $\eqref{s2}$ preserves the bi-Ha\-miltonian structure and is completely integrable.
It has solitary wave solutions \cite{dullin-gottwald-holm} as follow:
\[m_t = -B_2\frac{\delta E}{\delta m} = -B_1\frac{\delta F}{\delta m},\]
where
\[B_1 = \delta_x - \alpha^2\delta_x^3.\]
\[B_2 = \delta_x(m+\frac{c_0}{2}+(m+\frac{c_0}{2})\delta_x+\gamma \delta_x^3).\]
Important conservated quantities are:
\[E(u) = \frac{1}{2}\int_{\mathbb{R}}(u^2+\alpha^2u_x^2),\]
\[F(u) = \frac{1}{2}\int_{\mathbb{R}}(u^3+\alpha^3uu_x^2+c_0u^2-\gamma u_x^2).\]

In this paper we will also consider the two-component DGH equation that reads as follows:
\begin{equation}\label{s7}
\begin{cases} 
u_t - \alpha^2u_{txx} + c_0u_x + 3uu_x +\gamma u_{xxx}= \alpha^2(2u_xu_{xx}+uu_{xxx}) - \sigma \rho \rho_x, & t > 0, \quad x\in \mathbb{R}, \\ 
\rho_t+(u\rho)_x = 0  & t > 0, \quad x\in \mathbb{R}, \\
\rho(0,x) = \rho_0, \\
u(0,x) = u_0.
\end{cases}
\end{equation}

As shown by Constantin and Ivanov's \cite{cons-ivan}, system~\eqref{s7} can be derived from the shallow water theory. 
When $\alpha =1$ and $\gamma =0$, we get 
the two-component Camassa-Holm system.% \cite{li-olver}. 

From a geometrical meaning, the two-component DHG system corresponds to a geodesic flow on the semidirect
product Lie group of diffeomorphisms acting on densities, respected to the $H^1$  norm of velocity and the $L^2$ norm of the density. In an hydrodynamical context, 
we consider $\sigma =1$, and the natural boundary conditions are $u \rightarrow 0$ and $\rho \rightarrow 1$ as $x \rightarrow \infty$, for any $t$ . 
Let $\tilde{\rho} = \rho -1$, then $\tilde{\rho} \rightarrow 0$ as $x \rightarrow \infty$. 

In this case, we have:

\begin{equation}\label{s8}
\begin{cases} 
u_t - \alpha^2u_{txx} + c_0u_x + 3uu_x +\gamma u_{xxx}= \alpha^2(2u_xu_{xx}+uu_{xxx}) - \tilde{\rho}\tilde{\rho_x}-\tilde{\rho_x}, & t > 0, \quad x\in \mathbb{R}, \\ 
\tilde{\rho_t}+(u\tilde{\rho})_x+u_x = 0 & t > 0, \quad x\in \mathbb{R}, \\
\tilde{\rho}(0,x) = \tilde{\rho_0}, \\
u(0,x) = u_0.
\end{cases}
\end{equation}
System (\ref{s8}) has two Hamiltonians:

\[E(u) = \frac{1}{2}\int_{\mathbb{R}}(u^2+\alpha^2u_x^2+\rho^2),\]
\[F(u) = \frac{1}{2}\int_{\mathbb{R}}(u^3+\alpha^3uu_x^2+c_0u^2-\gamma u_x^2+2u\rho+u\rho^2).\]

As before,  \eqref{s8} is more conveniently reformulated as
\begin{equation}\label{s9}
\begin{cases}  
u_t + (u-\frac{\gamma}{\alpha^2})u_{x}= -\partial_xp*(\frac{\alpha^2}{2}u_x^2+u^2 + (c_0 + \frac{\gamma}{\alpha^2})u+\frac{1}{2}\tilde{\rho}^2+\tilde{\rho}) , & t > 0, \quad x\in \mathbb{R} \\ 
\tilde{\rho_t}+ u\tilde{\rho}_x  = -u_x\tilde{\rho}- u_x , & t > 0, \quad x\in \mathbb{R} \\
u(x,0) = u_0(x), & x\in \mathbb{R} \\
\tilde{\rho}(0,x) = \tilde{\rho_0}(x) & x\in \mathbb{R}
\end{cases} 
\end{equation}

The present paper adresses the problem of establishing wave breaking criteria for the DGH equation and the two-component DGH system. 
Previous results in this directions were obtained, {\it e.g.\/}, in \cite{zhai-gou-wang} for the two-component DGH system and in \cite{zhou} and \cite{liu} for DGH equation , where the authors dealt with  the special case $\gamma+c_0\alpha^2 = 0$ and $\alpha>0$ obtaining the finite time blowup of solution arising from certain initial profiles.
But the blowup conditions in the above mentioned papers involve  the computation of some global quantities associated with the initial datum (Sobolev norms, or integral conditions, or otherwise sign conditions, or antisymmetry relations, etc.).
Motivated by the recent paper \cite{brandolese}, we would like to establish a ``local-in-space'' blowup criterion, {\it i.e.\/} a criterion involving the properties of the initial datum only in a small neighborhood of a {\it single point\/}.
Such criterion will be more general (and more natural), than earlier blowup results. In addition, our approach will also go through when  
$\gamma+c_0\alpha^2$ is not necessarily zero. 

We will assume $\alpha > 0$. Notice that when $\alpha=0$ the DGH equation reduces to the KdV equation, for which the solutions exist globally and no blowup result can be obtained.
The following theorem represents our main result.

\begin{theorem}\label{theo:main}
Let $T^*$ be the maximal time of the unique solution $u \in C([0,T^*); H^s ) \cap C^1([0,T^*); H^{s-1} )$ of the 
Cauchy problem for the DGH equation~\eqref{s5}, arising from 
an initial datum $u_0 \in H^s(\mathbb{R}) $, with $s>\frac{3}{2}$. 
If there exists $x_0 \in \mathbb{R}$ such that:
\[
u'_0(x_0) < -\textstyle\frac{1}{\alpha}\bigl|u_0(x_0)+\textstyle\frac{1}{2}(c_0+\textstyle\frac{\gamma}{\alpha^2})\bigr|,
\]
then 
%there exist $t_0 >0$, finite such that $T^* \leq t_0 + \frac{4}{|u_x(t_0,q(t_0,x_0))|} < \infty$ 
$T^*<\infty$.
\end{theorem}

Our second result concerns the two-component DGH equation. It extends the recent blowup result~\cite{LvZhu} on the dispersionless two-component Camassa--Holm equations as well as several results therein quoted.
\begin{theorem}\label{theo:main2}
Suppose that $\gamma = 0$. Let $T^*$ be the maximal time of the unique solution $(u,\tilde{\rho})\in C([0,T^*); H^s \times H^{s-1}) \cap C^1([0,T^*); H^{s-1} \times H^{s-2})$ of the integrable two-component DGH system
~\eqref{s8}, starting from $(u_0,\tilde{\rho_0})\in H^s \times H^{s-1}$ with $s \geq \frac{5}{2}$. If there exist $x_0 \in \mathbb{R}$ such that:

\begin{itemize}
\item[(i)] $\tilde{\rho_0}(x_0) =-1$
\item[(ii)] $u'_0(x_0) < -\textstyle\frac{1}{\alpha}\bigl|u_0(x_0)+\textstyle\frac{1}{2}c_0\bigr|$,
\end{itemize}

then
%there exist $t_0 >0$, finite such that $T^* \leq t_0 + \frac{4}{|u_x(t_0,q(t_0,x_0))|} < \infty$
$T^*<\infty$.

\end{theorem}

\section{Preliminaries}
Using Kato's theory \cite{kato}, one can establish the local well-posedness theorem for the DGH equation (\ref{s1}) and the two-component 
DGH system (\ref{s7}). For example, the equation (\ref{s5}) can be rewritten as
\begin{equation}
\begin{cases}  
\frac{du}{dt}+A(u) = H(u) \\ 
u(x,0) = u_0(x), \\
\end{cases} 
\end{equation}
with $A(u) = (u+\lambda)u_x$ and $H(u) = -\partial_xp*(\frac{\alpha^2}{2}u_x^2+u^2 + (c_0 + \frac{\gamma}{\alpha^2})u)$. Several proofs of local existence thoery can be found in \cite{con-esc}, \cite{Li-Olver}, \cite{Mis}, \cite{Shk} for the DGH equation, and in \cite{gou-gao-liu}, \cite{zhu-xu} 
for the two-component DGH system. Here we recall the following result:
\begin{theorem}[See \cite{TiGuLi}] 
Given $u_0 \in H^s(\mathbb{R})$, $s > \frac{3}{2}$ and $\gamma+c_0\alpha^2 \geq 0$ of \eqref{s1}. Then there exists $T^* =T(\|u_0\|_{H^s}) > 0$ and a unique solution 
\[u \in C([0,T); H^s ) \cap C^1([0,T); H^{s-1})\]
of equation~\eqref{s5}.
Furthermore, the solution $u$ depends continuously on the initial data $u_0$.
\end{theorem}

\begin{theorem}[See \cite{zhu-xu}] 
Given $(u_0,\tilde{\rho_0})\in H^s \times H^{s-1}$, $s \geq \frac{5}{2}$ and $\gamma+c_0\alpha^2 \geq 0$ of \eqref{s8}. \\
Then there exists $T^* =T(\|u_0,\tilde{\rho_0}\|_{H^s\times H^{s-1}}) > 0$ and a unique solution 
\[(u,\tilde{\rho})\in C([0,T^*); H^s \times H^{s-1}) \cap C^1([0,T^*); H^{s-1} \times H^{s-2})\]
of the system~\eqref{s7}.
Furthermore, the solution $(u,\tilde{\rho})$ depends continuously on the initial data $(u_0,\tilde{\rho_0})$.
\end{theorem}

The maximal time of existence $T^*$ is known to be independent of the parameter $s$.
%The local-wellposedness theorem is applied for both periodic and non-period cases. 
Moreover, if $T^*$ is finite then $\lim_{t\rightarrow T}\|u(t)\|_{H^s} = \infty$. Next theorem tells something more about 
such blowup (or wave-breaking) scenario.

\begin{theorem}[See \cite{TiGuLi}] 
Given $u_0 \in H^s(\mathbb{R})$, $s > \frac{3}{2}$. Then the solution $u(t,x)$ of the DGH equation is uniformly bounded on [0,T). Moreover, a blowup occurs at the time $T < \infty$ if and only if 
\[\lim_{t \rightarrow T}\inf(\inf_{x \in \mathbb{R}}(u_x(t,x)))= -\infty\]
\end{theorem}

\begin{theorem}[See \cite{zhu-xu}] 
Given $(u_0,\tilde{\rho_0})\in H^s \times H^{s-1}$, $s \geq \frac{5}{2}$. Then the solution $u,\tilde{\rho}$ of the two-component DGH equation is uniformly bounded on [0,T). 
Moreover, a blowup occurs at the time $T < \infty$ if and only if 
\[\lim_{t \rightarrow T}\inf(\inf_{x \in \mathbb{R}}(u_x(t,x)))= -\infty\]
\end{theorem}

Following McKean's approach for the Camassa-Holm equation \cite{cons-moli}, we introduce 
the particle trajectory $q(t,x) \in C^1([0,T) \times \mathbb{R}, \mathbb{R})$, defined by

\begin{equation}\label{s6}
\begin{cases}  
q_t(t,x) = u(t,q(t,x)) - \frac{\gamma}{\alpha^2}, & t \in [0,T^*) \\ 
q(0,x) = x. \\
\end{cases}
\end{equation}
For every fixed $t \in [0,T)$, $q(t,.)$ is an increasing diffeomorphism of the real line. 
In fact, taking the derivative with respect to~$x$, yields
\[\frac{dq_t}{dx}=q_{xt}=u_x(t,q(t,x))q_x.\]  
Then
\[q_x(x,t) = \exp{\int_{0}^{t}u_x(q,s)}ds.\]

The momentum $m $ satisfies the fundamental identity,
\[m_0(x) + \frac{c_0}{2} + \frac{\gamma}{2\alpha^2} = \Bigl( m(t,q(t,x)) + \frac{c_0}{2} + \frac{\gamma}{2\alpha^2}\Bigr)q_x^2(t,x),\]
putting in evidence the specific properties of the momentum in the case ${c_0} + \frac{\gamma}{\alpha^2}=0$.

\section{Convolution estimates}
First of all, we state some useful results that will be used in the proof of  Theorem~\ref{theo:main}.
For convenience, let us set
\[\lambda = -\frac{\gamma}{\alpha^2}\qquad \mbox{and}\qquad 
k = \frac{1}{2}(c_0+\frac{\gamma}{\alpha^2}). 
\]
The following lemma generalizes the estimates in~\cite{brandolese}:
\begin{lemma}\label{l1}
With the above notations, we have the following inequalities, for all $u\in H^s(\R)$,
\begin{equation} \label{s3}
(p-\alpha \partial_xp)*\Bigl(\frac{\alpha^2}{2}u_x^2+ u^2 + 2ku\Bigr) \geq \frac{(u+k)^2}{2}-k^2
\end{equation}
and
\begin{equation} \label{s4}
(p+\alpha \partial_xp)*\Bigl(\frac{\alpha^2}{2}u_x^2+ u^2 + 2ku\Bigr) \geq \frac{(u+k)^2}{2}-k^2
\end{equation}
\end{lemma}

The above inequalities are sharp. The equality holds with the choice $u= ce^{-\frac{|x-y|}{\alpha}} - k$, 
with $c,y \in \mathbb{R}$.

\begin{proof}

We denote by $1_{\mathbb{R}^+}$ and $1_{\mathbb{R}^-}$ the characteristic functions  of 
$\mathbb{R}^+$ and $\mathbb{R}^-$ respectively. Hence:

\begin{eqnarray}
(p-\alpha \partial_xp)*(\frac{\alpha^2}{2}u_x^2+ u^2 + 2ku) &=& (p + \mbox{sign(x)}p)*(\frac{\alpha^2}{2}u_x^2+ u^2 + 2ku) \nonumber \\
&=& 2p1_{\mathbb{R}^+}*(\frac{\alpha^2}{2}u_x^2+ u^2 + 2ku) \nonumber \\
&=& \frac{1}{\alpha}\int_{-\infty}^{x}e^{\frac{y-x}{\alpha}}(\frac{\alpha^2}{2}u_x^2+ u^2 + 2ku)(y)dy \nonumber \\
&=& \frac{1}{\alpha}\int_{-\infty}^{x}e^{\frac{y-x}{\alpha}}(\frac{\alpha^2}{2}u_x^2+ (u+k)^2)(y)dy - \frac{k^2}{\alpha}\int_{-\infty}^{x}e^{\frac{y-x}{\alpha}}dy   \nonumber \\
&=&  \frac{1}{\alpha}\int_{-\infty}^{x}e^{\frac{y-x}{\alpha}}(\frac{\alpha^2}{2}u_x^2+ (u+k)^2)(y)dy - k^2. \nonumber 
\end{eqnarray}

Using Cauchy inequality, we have: 
\[
\begin{split}
\int_{-\infty}^{x}e^{\frac{y}{\alpha}}(\alpha^2 u_x^2+ (u+k)^2)(y)dy
&\geq 2\alpha  \int_{-\infty}^{x}e^{\frac{y}{\alpha}}(u+k)u_x(y)dy \\
&= \alpha e^{\frac{x}{\alpha}}(u+k)^2 -  \int_{-\infty}^{x}e^{\frac{y}{\alpha}}(u+k)^2(y)dy
\end{split}
\]

It follows that 
\[ \frac{1}{\alpha}\int_{-\infty}^{x}e^{\frac{y-x}{\alpha}}(\frac{\alpha^2}{2}u_x^2+ (u+k)^2)(y)dy \geq \frac{(u+k)^2}{2} \]

So, (\ref{s3}) is obtained. Similarly, one proves \eqref{s4}.

\end{proof}

\begin{lemma}\label{l2}
With the above definitions, we have the following inequality:
\begin{equation} \label{s7}
p*(\frac{\alpha^2}{2}u_x^2+ (u+k)^2) \geq \frac{(u+k)^2}{2}
\end{equation}
\end{lemma}

\begin{proof}
Observing that 
\[ p*(\frac{\alpha^2}{2}u_x^2+ (u+k)^2) = p1_{\mathbb{R}^+}*(\frac{\alpha^2}{2}u_x^2+ (u+k)^2) + p1_{\mathbb{R}^-}*(\frac{\alpha^2}{2}u_x^2+ (u+k)^2) \]

Following the lemma 3.1, we have:
\[p1_{\mathbb{R}^+}*(\frac{\alpha^2}{2}u_x^2+ (u+k)^2) \geq \frac{(u+k)^2}{4}\]
and
\[p1_{\mathbb{R}^-}*(\frac{\alpha^2}{2}u_x^2+ (u+k)^2) \geq \frac{(u+k)^2}{4}\]
The result is obtained

\end{proof}

Now, for any $f \in H^1(\mathbb{R})$, we have the obvious inequality:
\[ \min\{\alpha^2,1\}\|f\|_{H^1}^2 \leq \int_{\mathbb{R}}(f^2+\alpha^2f_x^2)dx \leq \max\{\alpha^2,1\}\|f\|_{H^1}^2\]

If we define:
\[\|f\|_{H_{\alpha}^1}^2 = \int_{\mathbb{R}}(f^2+\alpha^2f_x^2)dx, \]
then $H_{\alpha}^1 \subset L^{\infty}$. The next lemma  estimates the Sobolev costant related to the 
embedding  $H_{\alpha}^1 \subset L^{\infty}$ in $\mathbb{R}$

\begin{lemma}\label{l3}
We have the Sobolev embedding inequality:
\[\|u\|_{L^{\infty}} \leq \frac{1}{\sqrt{2\alpha}}\|u\|_{H_{\alpha}^1}.\]
The equality can be obtained. For example, we take $u= ce^{-\frac{|x-y|}{\alpha}}$ for some $c,y \in \mathbb{R}$.
\end{lemma}

\begin{proof}
For any $y \in \mathbb{R}$, we have:
\begin{eqnarray}
(u)^2(y) &=& \int_{-\infty}^{y}uu_xdx-\int_{y}^{+\infty}uu_xdx \nonumber \\
&\leq& \frac{1}{2\alpha}(\int_{-\infty}^{y}(u^2+\alpha^2u_x^2)dx + \int_{y}^{+\infty}(u^2+\alpha^2u_x^2)dx) \nonumber \\
&=& \frac{1}{2\alpha}\|u\|_{H_{\alpha}^1}^2.\nonumber \\
\end{eqnarray}
So, we obtain:
\[\|u\|_{L^{\infty}} \leq \frac{1}{\sqrt{2\alpha}}\|u\|_{H_{\alpha}^1}.\]
\end{proof}

\section{The local-in-space criterion for blow-up of the DGH equation}

Recall that, by definition, $(1-\alpha^2\partial_x^2)Qf = f$ for any $f \in L^2(\mathbb{R})$ 
so that $p*f-f = \alpha^2\partial_x^2(p*f)$. Taking the derivative with respect to $x$ in (\ref{s5}) yields:

$$\begin{cases}  
u_{tx}+(u+\lambda)u_{xx} = \frac{-u_x^2}{2}+\frac{u^2+2ku}{\alpha^2}-\frac{1}{\alpha^2}p*(\frac{\alpha^2}{2}u_x^2+u^2 + 2ku) \\ 
u(x,0) = u_0(x).
\end{cases} $$
For any $0 < T < T^*$, we see that $u$ and $u_x$ are continuous on $[0,T) \times \mathbb{R}$, and $u(t,x)$ is Lipschitz, uniformly respected to $t$. 
So, the flow map $q(t,x)$ introduced in (\ref{s6}):
\begin{equation}\label{flowl}\begin{cases}  
q_t(t,x) = u(t,q(t,x)) + \lambda \quad \quad \quad \quad t \in [0,T^*) \\ 
q(0,x) = x ,\\
\end{cases}\end{equation}
is indeed well defined in the interval $[0,T)$ and $q \in C^1([0,T)\times \mathbb{R}, \mathbb{R})$.

We have 
\begin{eqnarray}
\frac{d}{dt}[u_x(t,q(t,x))] &=&[u_{tx} + u_{xx}(u+\lambda)](t,q(t,x))  \nonumber\\
&= & \frac{-u_x^2}{2}+\frac{u^2+2ku}{\alpha^2}-\frac{1}{\alpha^2}p*(\frac{\alpha^2}{2}u_x^2+u^2 + 2ku) \nonumber \\
&= & \frac{-u_x^2}{2}+\frac{(u+k)^2}{\alpha^2}-\frac{1}{\alpha^2}p*(\frac{\alpha^2}{2}u_x^2+(u+k)^2), \nonumber 
\end{eqnarray}
where we used the fact that $\int_{\mathbb{R}}^{}p(t,x) =1$. 
By lemma (\ref{l2}),  $\frac{1}{\alpha^2}p*(\frac{\alpha^2}{2}u_x^2+(u+k)^2) \geq \frac{(u+k)^2}{2\alpha^2}$. \
Hence,
\begin{eqnarray}
\frac{d}{dt}[u_x(t,q(t,x))] &\leq & (-\frac{1}{2}u_x^2 +\frac{(u+k)^2}{\alpha^2} - \frac{(u+k)^2}{2\alpha^2})(t,q(t,x)) \nonumber\\
&=& (-\frac{1}{2}u_x^2+ \frac{1}{2\alpha^2}(u+k)^2)(t,q(t,x)) \nonumber 
\end{eqnarray}

Inspired by~\cite{brandolese}, we now introduce 
\[A(t,x) = e^{\frac{q(t,x)}{\alpha}+\frac{(k-\lambda)t}{\alpha}}(\frac{1}{\alpha}(u+k)-u_x)(t,q(t,x))\]
and 
\[B(t,x) = e^{-\frac{q(t,x)}{\alpha}+\frac{(-k+\lambda)t}{\alpha}}(\frac{1}{\alpha}(u+k)+u_x)(t,q(t,x)).\]
Then,
\[\frac{d}{dt}[u_x(t,q(t,x))]  \leq\frac{1}{2}(AB)(t,q(t,x)) \]

The following result plays an important role:
\begin{lemma}
For all $x \in \mathbb{R}$,  the map $A(t,x)$ is monotonically increasing and $B(t,x)$
are monotonically decreasing with respect to $t\in [0,T]$.
\end{lemma}

\begin{proof}
Let us  calculate $\frac{d}{dt}A(t,x)$:
\begin{eqnarray}
\frac{d}{dt}A(t,x) &=& e^{\frac{q(t,x)}{\alpha}+\frac{(k-\lambda)t}{\alpha}}(\frac{1}{\alpha}(u+\lambda)(\frac{1}{\alpha}(u+k) - u_x) -\frac{k-\lambda}{\alpha}(\frac{1}{\alpha}(u+k) - u_x) \nonumber \\
&& \qquad\qquad \qquad+ (\frac{1}{\alpha}(u+k) -u_x)_t +(u+\lambda)(\frac{1}{\alpha}(u+k) - u_x)_x) \nonumber \\
&=& e^{\frac{q(t,x)}{\alpha}+\frac{(k-\lambda)t}{\alpha}}(\frac{u+k}{\alpha}(\frac{1}{\alpha}(u+k) - u_x)+ \frac{1}{\alpha}(u_t + (u+\lambda)u_x) - (u_{xt} + (u+\lambda)u_{xx})) \nonumber \\
&=& e^{\frac{q(t,x)}{\alpha}+\frac{(k-\lambda)t}{\alpha}}(\frac{u+k}{\alpha}(\frac{1}{\alpha}(u+k) - u_x) - \frac{1}{\alpha}\partial_xp*(\frac{\alpha^2}{2}u_x^2+u^2+2ku) +\frac{u_x^2}{2}-\frac{u^2+2ku}{\alpha^2} \nonumber \\
&& \qquad\qquad + \frac{1}{\alpha^2}p*(\frac{\alpha^2}{2}u_x^2+u^2+2ku)) \nonumber \\
&=& e^{\frac{q(t,x)}{\alpha}+\frac{(k-\lambda)t}{\alpha}}(\frac{u_x^2}{2}-\frac{(u+k)u_x}{\alpha}+ \frac{k^2}{\alpha^2}+ \frac{1}{\alpha^2} (p-\alpha \partial_xp)*(\frac{\alpha^2}{2}u_x^2+u^2+2ku).\nonumber 
\end{eqnarray}
By lemma (\ref{l1}),  $(p-\alpha \partial_xp)*(\frac{\alpha^2}{2}u_x^2+u^2+2ku)  \geq \frac{(u+k)^2}{2}-k^2$. 
Then we deduce that 
\begin{eqnarray}
\frac{d}{dt}A(t,x) &\geq& e^{\frac{q(t,x)}{\alpha}+\frac{(k-\lambda)t}{\alpha}}(\frac{u_x^2}{2}-\frac{(u+k)u_x}{\alpha}+ \frac{k^2}{\alpha^2} + \frac{1}{\alpha^2}(\frac{(u+k)^2}{2}-k^2)) \nonumber \\
&=&  e^{\frac{q(t,x)}{\alpha}+\frac{(k-\lambda)t}{\alpha}}(\frac{u_x^2}{2}-\frac{(u+k)u_x}{\alpha}+\frac{(u+k)^2}{2\alpha^2})  \nonumber \\
% &\geq& 0 .\nonumber 
\end{eqnarray}
So, for all $x$, $t \mapsto A(t,x)$ is monotonically increasing. 
Similarly
\begin{eqnarray}
\frac{d}{dt}B(t,x) &=& e^{-\frac{q(t,x)}{\alpha}+\frac{(-k+\lambda)t}{\alpha}}(-\frac{1}{\alpha}(u+\lambda)(\frac{1}{\alpha}(u+k) + u_x)+ \frac{-k+\lambda}{\alpha}(\frac{1}{\alpha}(u+k) + u_x)  \nonumber \\
&& \qquad\qquad \qquad+ (\frac{1}{\alpha}(u+k) + u_x)_t+(u+\lambda)(\frac{1}{\alpha}(u+k) + u_x)_x) \nonumber \\
&=&e^{-\frac{q(t,x)}{\alpha}+\frac{(-k+\lambda)t}{\alpha}}(-\frac{u+k}{\alpha}(\frac{1}{\alpha}(u+k) + u_x)+ \frac{1}{\alpha}(u_t + (u+\lambda)u_x) + (u_{xt} + (u+\lambda)u_{xx})) \nonumber \\
&=& e^{-\frac{q(t,x)}{\alpha}+\frac{(-k+\lambda)t}{\alpha}}(-\frac{u+k}{\alpha}(\frac{1}{\alpha}(u+k) + u_x) - \frac{1}{\alpha}\partial_xp*(\frac{\alpha^2}{2}u_x^2+u^2+2ku) -\frac{u_x^2}{2}+\frac{u^2+2ku}{\alpha^2} \nonumber \\
&& \qquad\qquad- \frac{1}{\alpha^2}p*(\frac{\alpha^2}{2}u_x^2+u^2+2ku)) \nonumber \\
&=& -e^{-\frac{q(t,x)}{\alpha}+\frac{(-k+\lambda)t}{\alpha}}(\frac{u_x^2}{2}+\frac{(u+k)u_x}{\alpha}+ \frac{k^2}{\alpha^2}+\frac{1}{\alpha^2}(p+\alpha \partial_xp)*(\frac{\alpha^2}{2}u_x^2+u^2+2ku)) .\nonumber 
\end{eqnarray}

Applying now the second estimate of  Lemma (\ref{l1}) we obtain:
\begin{eqnarray}
\frac{d}{dt}B(t,x) &\leq& -e^{-\frac{q(t,x)}{\alpha}+\frac{(-k+\lambda)t}{\alpha}}(\frac{u_x^2}{2}+\frac{(u+k)u_x}{\alpha}+ \frac{k^2}{\alpha^2}+\frac{(u+k)^2}{2}-\frac{k^2}{\alpha^2}) \nonumber \\
&=& -e^{-\frac{q(t,x)}{\alpha}+\frac{(-k+\lambda)t}{\alpha}}(\frac{u_x^2}{2}+\frac{(u+k)u_x}{\alpha}+\frac{(u+k)^2}{2\alpha^2})  \nonumber \\
&\leq& 0 \nonumber
\end{eqnarray}
So, for all $x$, $t \mapsto B(t,x)$ is monotonically decreasing
\end{proof}

Now, we are ready to prove the main theorem.

\begin{proof}[Proof of~Theorem~\ref{theo:main}]

Let $x_0$ be such that $u'_0(x_0) < -\frac{1}{\alpha}|u_0(x_0)+\frac{1}{2}(c_0+\frac{\gamma}{\alpha^2})| = - \frac{1}{\alpha}|u_0(x_0)+k| $. 
We denote 
\[g(t) = u_x(t,q(t,x_0)),\] 
$A(t) = A(t,x_0)$ and $B(t) = B(t,x_0)$. For all $t \in [0,T^*)$, we have:
$\frac{d}{dt}A(t) \geq 0$ and $\frac{d}{dt}B(t)  \leq 0$.

So, $A(t) \geq A(0) = e^{\frac{x_0}{\alpha}}(\frac{1}{\alpha}(u_0(x_0)+k)-u_0'(x_0)) > 0$, and $B(t) \leq B(0) = e^{\frac{-x_0}{\alpha}}(\frac{1}{\alpha}(u_0(x_0)+k)+u_0'(x_0)) < 0$.
Thus $AB(t) \leq AB(0) < 0$. Then for all $t \in [0,T)$: $g'(t) < 0$.  \\

Assume by contradiction $T^* = \infty$, then $g(t) \leq g(0) - \alpha_0t$ where $\alpha_0 = \frac{1}{2}(u'(0)^2-\frac{1}{\alpha^2}(u_0+k)^2)(x_0)$.
We choose $t_0$ such that $g(0) - \alpha_0t_0 \leq 0$ and $(g(0)-\alpha_0t_0)^2 \geq \frac{1}{\alpha^3}(\|u_0\|_{H_{\alpha}^1}+k\sqrt{2\alpha})^2 $. For $t \geq t_0$, we have: 
\begin{eqnarray}
g'(t) &\leq& \frac{1}{2}(\frac{1}{\alpha^2}(u+k)^2 - u_x^2)(t,q(t, x_0)) \nonumber \\
&\leq& \frac{1}{2}(\frac{1}{\alpha^2}(\|u(t,.)\|_{L^{\infty}}+k)^2 - g(t)^2).\nonumber 
\end{eqnarray}

Using Sobolev embedding inequality in lemma (\ref{l3}) and the energy conservation identity, one has:
\begin{eqnarray}
g'(t) &\leq&  \frac{1}{2}(\frac{1}{\alpha^2}(\frac{\|u_0\|_{H_{\alpha}^1}}{\sqrt{2\alpha}}+k)^2 -g(t)^2)\nonumber \\
&\leq& -\frac{1}{4}g(t)^2 \nonumber 
\end{eqnarray}
for all $t \in (t_0, \infty)$. Dividing both sides by $g^2(t)$ and intergrating, we get 
\[\frac{1}{g(t_0)} - \frac{1}{g(t)}+\frac{1}{4}(t-t_0) \leq 0 \quad \quad t \geq t_0.\]
This is a contradiction since $- \frac{1}{g(t)} > 0$ and $\frac{1}{4}(t-t_0) \rightarrow \infty$ as $t \rightarrow \infty$. 
Thus $u_x(t,q(t,x_0)) $ blow up in finite time and $T^* \leq t_0 + \frac{4}{|g(t_0)|} < \infty$.

\end{proof}

\subsection*{Remark} Local-in-space blowup criterion in the particular case  $\gamma= c_0 =0$ and $\alpha=1$  (corresponding to  the Camassa-Holm equation) has been first built in \cite{brandolese} and later extended in~\cite{brandolese-cortez1} to a class of possibly non-quadratic nonlinearities. See also \cite{brandolese-cortez2} for improvements specific to the periodic case.
Our Theorem~\ref{theo:main} improves the result of~\cite{brandolese} in a different direction, by extending the blowup result to arbitrary values of  $\gamma$, $c_0$ and $\alpha>0$.

\section{Blow-up for two-component DGH system}
When $\gamma = 0$, the equation \ref{s8} becomes
\begin{equation}\label{s10}
\begin{cases} 
u_t - \alpha^2u_{txx} + c_0u_x + 3uu_x = \alpha^2(2u_xu_{xx}+uu_{xxx}) - \tilde{\rho}\tilde{\rho_x}-\tilde{\rho_x}, & t > 0, \quad x\in \mathbb{R}, \\ 
\tilde{\rho_t}+(u\tilde{\rho})_x+u_x = 0 & t > 0, \quad x\in \mathbb{R}, \\
\tilde{\rho}(0,x) = \tilde{\rho_0}, \\
u(0,x) = u_0.
\end{cases}
\end{equation}
This can be rewritten as
\begin{equation}\label{s11}
\begin{cases}  
u_t + uu_{x}= -\partial_xp*(\frac{\alpha^2}{2}u_x^2+u^2 + c_0u+\frac{1}{2}\tilde{\rho}^2+\tilde{\rho}) , & t > 0, \quad x\in \mathbb{R} \\ 
\tilde{\rho_t}+ u\tilde{\rho}_x  = -u_x\tilde{\rho}- u_x , & t > 0, \quad x\in \mathbb{R} \\
u(x,0) = u_0(x), & x\in \mathbb{R} \\
\tilde{\rho}(0,x) = \tilde{\rho_0}(x) & x\in \mathbb{R}
\end{cases} 
\end{equation}
Here we give the proof for theorem~\ref{theo:main2}
\begin{proof}

Again, using the identity $p*f-f = \alpha^2\partial_x^2(p*f)$, we take the derivative with respect to~$x$ in (\ref{s9}) yields:
\[u_{tx}+uu_{xx} =  \frac{-u_x^2}{2}+\frac{u^2+c_0u}{\alpha^2}+\frac{1}{\alpha^2}(\frac{\tilde{\rho}^2}{2}+\tilde{\rho})-\frac{1}{\alpha^2}p*(\frac{\alpha^2}{2}u_x^2+u^2 + c_0u+\frac{\tilde{\rho}^2}{2}+\tilde{\rho})\]

As before, we make use of the flow map, defined as in~\eqref{flowl}. When $\gamma = 0$, the map becomes
\begin{equation}\label{flowl2}
\begin{cases}  
q_t(t,x) = u(t,q(t,x)) \quad \quad \quad \quad t \in [0,T^*) \\ 
q(0,x) = x.\\
\end{cases}
\end{equation}
Notice that $q \in C^1([0,T)\times \mathbb{R}, \mathbb{R})$.  
We have
\begin{eqnarray}
\frac{d}{dt}[u_x(t,q(t,x))] &=&[u_{tx} + uu_{xx}](t,q(t,x))  \nonumber\\
&= & \frac{-u_x^2}{2}+\frac{u^2+c_0u}{\alpha^2}+\frac{1}{\alpha^2}(\frac{\tilde{\rho}^2}{2}+\tilde{\rho})-\frac{1}{\alpha^2}p*(\frac{\alpha^2}{2}u_x^2+u^2 + c_0u+\frac{\tilde{\rho}^2}{2}+\tilde{\rho}) \nonumber \\
&= & \frac{-u_x^2}{2}+\frac{(u+\frac{c_0}{2})^2}{\alpha^2}+\frac{1}{2\alpha^2}(\tilde{\rho}+1)^2-\frac{1}{\alpha^2}p*(\frac{\alpha^2}{2}u_x^2+(u+\frac{c_0}{2})^2+(\tilde{\rho}+1)^2). \nonumber 
\end{eqnarray}
Applying Lemma~\ref{l2}:  $\frac{1}{\alpha^2}p*(\frac{\alpha^2}{2}u_x^2+(u+\frac{c_0}{2})^2) \geq \frac{(u+\frac{c_0}{2})^2}{2\alpha^2}$, and the obvious estimate $p*(\tilde{\rho}+1)^2 \geq 0$, we get
\begin{eqnarray}
\frac{d}{dt}[u_x(t,q(t,x))] &\leq & (-\frac{1}{2}u_x^2 +\frac{(u+\frac{c_0}{2})^2}{\alpha^2}+\frac{1}{2\alpha^2}(\tilde{\rho}+1)^2 - \frac{(u+\frac{c_0}{2})^2}{2\alpha^2})(t,q(t,x)) \nonumber\\
&=& (-\frac{1}{2}u_x^2+ \frac{1}{2\alpha^2}(u+\frac{c_0}{2})^2+\frac{1}{2\alpha^2}(\tilde{\rho}+1)^2)(t,q(t,x)). \nonumber 
\end{eqnarray}

We also have the following identity:
\[
\begin{split}
\frac{d}{dt}[(\tilde{\rho}&(t,q(t,x))+1)q_x(t,x)]  \\
&=  (\tilde{\rho_t}(t,q(t,x))+\tilde{\rho_x}(t,q(t,x))q_t(t,x))q_x(t,x)+(\tilde{\rho}(t,q(t,x))+1)q_{xt}(t,x) \\
&= \Bigl(\tilde{\rho_t}(t,q(t,x))+\tilde{\rho_x}(t,q(t,x))u(t,q(t,x))+\tilde{\rho}(t,q(t,x))u_x(t,q(t,x))+u_x(t,q(t,x))\Bigr)q_x(t,x) \\
&= 0.
\end{split}
\]

This implies that $(\tilde{\rho}(t,q(t,x))+1)q_x(t,x) = (\tilde{\rho_0}(x)+1) $.
The initial condition implies $\tilde{\rho_0}(x_0)+1 = 0$,
then $\tilde{\rho}(t,q(t,x_0))+1=0$  for all $t$. Therefore, \\
\begin{eqnarray}
\frac{d}{dt}[u_x(t,q(t,x_0))] &\leq&  (-\frac{1}{2}u_x^2+ \frac{1}{2\alpha^2}(u+\frac{c_0}{2})^2)(t,q(t,x_0))  \nonumber
\end{eqnarray}

Using now similar calculations as for the DHG equation, we factorize  
$(-\frac{1}{2}u_x^2+ \frac{1}{2\alpha^2}(u+\frac{c_0}{2})^2)(t,q(t,x_0)) = \frac{1}{2}(AB)(t,q(t,x_0))$
where $A(t,x_0) = e^{\frac{q(t,x_0)}{\alpha}+\frac{c_0t}{2\alpha}}(\frac{1}{\alpha}(u+\frac{c_0}{2})-u_x)(t,q(t,x_0))$ and $B(t,x_0) = e^{-\frac{q(t,x_0)}{\alpha}-\frac{c_0t}{2\alpha}}(\frac{1}{\alpha}(u+\frac{c_0}{2})+u_x)(t,q(t,x_0))$. 
Using Lemma~4.2, we see that  $A(t,x_0)$ is monotically increasing and $B(t,x_0)$
monotonically decreasing with respect to~$t$.

But at $x_0$ we have, by our assumption, $u'_0(x_0) < -\frac{1}{\alpha}|u_0(x_0)+\frac{c_0}{2}| $. 
We denote $g(t) = u_x(t,q(t,x_0))$, $A(t) = A(t,x_0)$ and $B(t) = B(t,x_0)$. For all $t \in [0,T^*)$, we have:

\[\frac{d}{dt}A(t) \geq 0 \]
and 
\[\frac{d}{dt}B(t)  \leq 0  \]

So, $A(t) \geq A(0) = e^{\frac{x_0}{\alpha}}(\frac{1}{\alpha}(u_0(x_0)+\frac{c_0}{2})-u_0'(x_0)) > 0$, and $B(t) \leq B(0) = e^{\frac{-x_0}{\alpha}}(\frac{1}{\alpha}(u_0(x_0)+\frac{c_0}{2})+u_0'(x_0)) < 0$.
Thus $AB(t) \leq AB(0) < 0$. Then for all $t \in [0,T)$: $g'(t) < 0$.  \\

Assume by contradiction $T^* = \infty$, then $g(t) \leq g(0) - \alpha_0t$ where $\alpha_0 = \frac{1}{2}(u'(0)^2-\frac{1}{\alpha}(u_0+\frac{c_0}{2})^2)(x_0)$.
We choose $t_0$ such that $g(0) - \alpha_0t_0 \leq 0$ and $(g(0)-\alpha_0t_0)^2 \geq \frac{1}{\alpha^3}(\|u_0\|_{H_{\alpha}^1}+\|\tilde{\rho}_0\|_{L^2}+\frac{c_0}{2}\sqrt{2\alpha})^2 $. For $t \geq t_0$, we have: 

\begin{eqnarray}
g'(t) &\leq& \frac{1}{2}(\frac{1}{\alpha^2}(u+\frac{c_0}{2})^2 - u_x^2)(t,q(t, x_0)) \nonumber \\
&\leq& \frac{1}{2}(\frac{1}{\alpha^2}(\|u(t,.)\|_{L^{\infty}}+\frac{c_0}{2})^2 - g(t)^2)\nonumber 
\end{eqnarray}

Using Sobolev embedding inequality in lemma (\ref{l3}) and the energy conservation identity, one has:
\begin{eqnarray}
g'(t) &\leq&  \frac{1}{2}(\frac{1}{\alpha^2}(\frac{\|u\|_{H_{\alpha}^1}}{\sqrt{2\alpha}}+\frac{c_0}{2})^2 -g(t)^2)\nonumber \\
&\leq&  \frac{1}{2}(\frac{1}{\alpha^2}(\frac{\|u\|_{H_{\alpha}^1}}{\sqrt{2\alpha}}+\frac{\|\tilde{\rho}\|_{L^2}}{\sqrt{2\alpha}}+\frac{c_0}{2})^2 -g(t)^2)\nonumber \\
&=&  \frac{1}{2}(\frac{1}{\alpha^2}(\frac{\|u_0\|_{H_{\alpha}^1}}{\sqrt{2\alpha}}+\frac{\|\tilde{\rho}_0\|_{L^2}}{\sqrt{2\alpha}}+\frac{c_0}{2})^2 -g(t)^2)\nonumber \\
&\leq& -\frac{1}{4}g(t)^2 \nonumber 
\end{eqnarray}

for all $t \in (t_0, \infty)$. Dividing both sides by $g^2(t)$ and intergrating, we get 
\[\frac{1}{g(t_0)} - \frac{1}{g(t)}+\frac{1}{4}(t-t_0) \leq 0 \quad \quad t \geq t_0\]
This is a contradiction since $- \frac{1}{g(t)} > 0$ and $\frac{1}{4}(t-t_0) \rightarrow \infty$ as $t \rightarrow \infty$. 
Thus $u_x(t,q(t,x_0)) $ blow up in finite time and $T^* \leq t_0 + \frac{4}{|g(t_0)|} < \infty$

\end{proof}

\maketitle

\section*{Appendix}
We give another proof for theorem (\ref{theo:main}). A similar technique can be applied to Theorem~\ref{theo:main2}.

\begin{proof}
Using the argument: 
\[\frac{d}{dt}[u_x(t,q(t,x))] \leq  (-\frac{1}{2}u_x^2+ \frac{1}{2\alpha^2}(u+k)^2)(t,q(t,x)) \]

We factorize  $(-\frac{1}{2}u_x^2+ \frac{1}{2\alpha^2}(u+k)^2)(t,q(t,x)) = \frac{1}{2}(AB)(t,q(t,x))$ where $ A(t,x) = \frac{1}{\alpha}(u+k) - u_x$ and 
$B(t,x) = \frac{1}{\alpha}(u+k) + u_x $. So, it follows that:
\[\frac{d}{dt}[u_x(t,q(t,x))] \leq \frac{1}{2}(AB)(t,x)\]

We have:

\begin{eqnarray}
\frac{d}{dt}A(t,x) &=& \frac{d}{dt}(\frac{u+k}{\alpha}-u_x)(t,q(t,x)) \nonumber \\
&=& (\frac{u_t}{\alpha}-u_{xt})+(\frac{u_x}{\alpha}-u_{xx})q_t(t,x) \nonumber \\
&=& (\frac{u_t}{\alpha}-u_{xt})+(\frac{u_x}{\alpha}-u_{xx})(u+\lambda) \nonumber \\
&=& \frac{1}{\alpha}(u_t+(u+\lambda)u_x)-(u_{tx}+(u+\lambda)u_{xx}) \nonumber \\
&=& \frac{u_x^2}{2}-\frac{u^2+2ku}{\alpha^2}+\frac{1}{\alpha^2} (p-\alpha \partial_xp)*(\frac{\alpha^2}{2}u_x^2+u^2+2ku)  \nonumber 
\end{eqnarray}

By lemma (\ref{l1}):  $(p-\alpha \partial_xp)*(\frac{\alpha^2}{2}u_x^2+u^2+2ku)  \geq \frac{(u+k)^2}{2}-k^2$, then we have:

\begin{eqnarray}
\frac{d}{dt}A(t,x) &\geq& \frac{u_x^2}{2}-\frac{u^2+2ku}{\alpha^2} +\frac{1}{\alpha^2}(\frac{(u+k)^2}{2}-k^2) \nonumber \\
&=&\frac{u_x^2}{2}-\frac{u^2+2ku}{\alpha^2} \nonumber \\
&=&-\frac{1}{2}(AB)(t,x) \nonumber 
\end{eqnarray}

Similarly, computing for $B(t,x)$ yields

\begin{eqnarray}
\frac{d}{dt}B(t,x) &=& \frac{d}{dt}(\frac{u+k}{\alpha}+u_x)(t,q(t,x)) \nonumber \\
&=& (\frac{u_t}{\alpha}+u_{xt})+(\frac{u_x}{\alpha}+u_{xx})q_t(t,x) \nonumber \\
&=& (\frac{u_t}{\alpha}+u_{xt})+(\frac{u_x}{\alpha}+u_{xx})(u+\lambda) \nonumber \\
&=& \frac{1}{\alpha}(u_t+(u+\lambda)u_x)+(u_{tx}+(u+\lambda)u_{xx}) \nonumber \\
&=& -\frac{u_x^2}{2}+\frac{u^2+2ku}{\alpha^2}-\frac{1}{\alpha^2} (p+\alpha \partial_xp)*(\frac{\alpha^2}{2}u_x^2+u^2+2ku)  \nonumber 
\end{eqnarray}

By lemma (\ref{l1}):  $(p+\alpha \partial_xp)*(\frac{\alpha^2}{2}u_x^2+u^2+2ku)  \geq \frac{(u+k)^2}{2}-k^2$, then we have:

\begin{eqnarray}
\frac{d}{dt}B(t,x) &\leq& -\frac{u_x^2}{2}+\frac{u^2+2ku}{\alpha^2} -\frac{1}{\alpha^2}(\frac{(u+k)^2}{2}-k^2) \nonumber \\
&=&-\frac{u_x^2}{2}+\frac{u^2+2ku}{\alpha^2} \nonumber \\
&=&\frac{1}{2}(AB)(t,x) \nonumber 
\end{eqnarray}

The initial condition $u'_0(x_0) < -\textstyle\frac{1}{\alpha}\bigl|u_0(x_0)+\textstyle\frac{1}{2}(c_0+\textstyle\frac{\gamma}{\alpha^2})\bigr|$
is equivalent to $A(0,x_0) > 0$ and $B(0,x_0) < 0$. Let:

\begin{eqnarray}
\omega = \sup\{t \in [0,T^*): A(.,x_0) > 0 \quad \mbox{and} \quad B(.,x_0) < 0  \quad \mbox{on} \quad  [0,t] \} \nonumber 
\end{eqnarray}

Then $\omega > 0$. If $\omega < T^*$ then at least one of the inequalies $A(\omega,x_0) \leq 0$ and $B(\omega,x_0) \geq 0$ 
must be true. This is a contradiction
with the fact that $AB(.,x_0) < 0$ on the interval $[0,\omega ]$, then $A(\omega,x_0) \geq A(0,x_0) >0$ and $B(\omega, x_0) \leq B(0,x_0) < 0$.
Hence, $\omega = T^*$ . \\

To conclude the proof, we argue as in~\cite{brandolese-cortez2}, considering 
\[h(t) = \sqrt{-(AB)(t,x_0)}\]
Then the time derivative of $h$

\begin{eqnarray}
\frac{d}{dt}h(t) &=& -\frac{A_tB+AB_t}{2\sqrt{-AB}}(t,x_0) \nonumber \\
&\geq& \frac{(-AB)(A-B)}{4\sqrt{-AB}}(t,x_0)\nonumber 
\end{eqnarray}

By the geometric-arithmetic mean inequality $(A-B)(t,x_0) \geq 2\sqrt{-(AB)(t,x_0)} = 2h(t)$, it follows

\[\frac{d}{dt}h(t) \geq \frac{1}{2}h^2(t)\]

But $h(0)= \sqrt{-(AB)(0,x_0)} > 0$. Hence the solution blows up in finite time and $T^* < \frac{2}{h(0)}$. Or it can be rewrire as:
\[T^* < \frac{2}{\sqrt{u'_0(x_0)^2 - \frac{1}{\alpha^2}(u_0(x_0)+k)^2}}\]

\end{proof}

We conclude observing that we do not know if Theorem~\ref{theo:main2}
 remains valid when $\gamma\not=0$. The main difficulty arises from the fact that
when $\gamma\not=0$, the underlining nonlinear transport equations associated with $u$ and $\rho$ travel with different speed (the two speeds are $u-\gamma/\alpha^2$ and $u$ respectively). This makes difficult to use the characteristics method to derive the ordinary differential system leading to the blowup.

\end{document}